\documentclass{amsart}
\usepackage{amsfonts,amssymb,amsmath,amsthm}
\usepackage{url}
\usepackage{enumerate}

\newtheorem{theorem}{Theorem}[section]
\newtheorem{lemma}[theorem]{Lemma}
\newtheorem{proposition}[theorem]{Proposition}
\newtheorem{corollary}[theorem]{Corollary}
\theoremstyle{definition}

\newtheorem{example}[theorem]{Example}
\numberwithin{equation}{section}

\newcommand{\Pp}{{\mathbb P}}

\newcommand{\R}{{\mathbb R}}
\newcommand{\C}{{\mathbb C}}
\newcommand{\cnd}{C_{n,d}}
\newcommand{\barcnd}{\bar{C}_{n,d}}
\newcommand{\nt}{\text{NT}}
\newcommand{\cnt}{\text{cNT}}
\newcommand{\gnd}{\gamma_{n,d}}
\newcommand{\bargnd}{\bar{\gamma}_{n,d}}
\newcommand{\cn}{C_n}

\newcommand{\Z}{{\mathbb{Z}}}
\newcommand{\codim}{\mathrm{codim}}
\let\phi\varphi
\newcommand{\varnt}{\boldsymbol F^{\nt_d}}

\author{Joachim von zur Gathen}
\address{
    B-IT\\
    Universit\"{a}t Bonn\\
    D - 53113 Bonn}
\email{gathen@bit.uni-bonn.de}
\thanks{JvzG acknowledges the support of the B-IT Foundation and the Land
Nordrhein-Westfalen}

\author{Guillermo Matera}
\address{
Instituto del Desarrollo Humano\\
Universidad Nacional de Gene\-ral Sarmiento, J.M. Guti\'errez 1150
(B1613GSX) Los Polvorines, Buenos Aires, Argentina \\
and National Council of Science and Technology (CONICET),
Ar\-gentina} \email{gmatera@ungs.edu.ar}
\thanks{GM is partially supported by the grants PIP CONICET 11220130100598,
PIO {CONICET-UNGS} 14420140100027 and UNGS 30/3084.}

\begin{document}

\title[Decomposable polynomials over $\R$ and $\C$]{Density of real and complex decomposable univariate polynomials}

\begin{abstract}
We estimate the density of tubes around the algebraic variety of
decomposable univariate polynomials over the real and the complex
numbers.
\end{abstract}
\maketitle

For two univariate polynomials $g,h \in F[x]$ of degrees $d$, $e$,
respectively, over a field $F$, their \emph{composition}
\begin{equation}
\label{comp} f = g(h) = g \circ h \in F[x]
\end{equation}
is a polynomial of degree $n = de$. If such $g$
 and $h$ exist with degree at least $2$, then $f$ is called
 \emph{decomposable} (or \emph{composed, composite}, a \emph{composition}).

 Since the foundational work of Ritt, Fatou, and Julia in the 1920s on
 compositions over $\C$, a substantial body of work has been concerned
 with structural properties (e.g., \cite{frimac69},
 \cite{dorwha74}, \cite{sch82c, sch00c}, \cite{zan93}), with
 algorithmic questions (e.g., \cite{barzip85}, \cite{kozlan89b},
 \cite{bla13}), and with enumeration over finite fields, exact and
 approximate (e.g., \cite{gie88b}, \cite{gatgie10},
 \cite{blagat13}, \cite{gat08c}, \cite{zie14}).

This paper presents analogs for the case of the real or complex
numbers of the latter counting results. What does counting mean
here? The dimensions and degrees of its irreducible components as
algebraic varieties? These quantities turn up in our argument, but
we bound here the \emph{density} of these components. Any proper
algebraic subvariety $X$ of $\R^{n}$ has volume and density $0$.
However, we can bump up the dimension of $X$ to $n$ by taking an
$\epsilon$-tube $U_{\epsilon}$ around $X$, replacing each point in
$X$ by a hypercube $(- \epsilon, \epsilon)^{\codim X}$ for ``small''
positive $\epsilon$. If this is done properly, $U_{\epsilon}$ has
dimension $n$. Its volume may be infinite, and we make it finite by
intersecting with a hypercube $(-B, B)^{n}$ for some ``large''
positive $B$. Then the \emph{density} of $X$ in $(-B, B)^{n}$ is
this finite volume divided by the volume $(2B)^{n}$ of the large
hypercube. A similar approach works in the complex case.

Let $X$ be an equidimensional real or complex algebraic variety
embedded in a $k$-dimensional affine space with codimension $m$, and
consider the $\epsilon$-neighborhood $\{y \colon \vert x-y \vert <
\epsilon \}$ of some point $x$ in the space.  This is a real
hypercube or a complex polycylinder, respectively. We also use the
corresponding notion for a projective space. There are several
notions for forming an $\epsilon$-tube around $X$, namely, as the
union of all
\begin{itemize}
\item $k$-dimensional $\epsilon$-neighborhoods of $x \in X$,
\item $m$-dimensional $\epsilon$-neighborhoods in the direction normal
  to $x \in X$,
\item $m$-dimensional $\epsilon$-neighborhoods in a fixed direction
  around $x \in X$.
\end{itemize}
Singular points may be disregarded.  The first two tubes comprise
the points in the affine space whose distance to $X$ in a normal
direction is less than $\epsilon$.  Thus the two notions coincide,
at least for smooth varieties. The ratio to the volume in the third
notion is locally $~cos \alpha$, where $\alpha$ is the angle between
the two $m$-dimensional linear spaces, namely the normal space (at a
nonsingular point) and the space in the chosen fixed direction.

The present paper uses exclusively the third notion, the others are
included here only for perspective.

Weyl \cite{wey39}, answering a question posed by Hotelling
\cite{hot39}, proved fundamental results on tubes around manifolds.
Since then, the topic has been studied in topology and differential
geometry and is the subject of the textbook of Gray \cite{gra90a},
which includes many further references. The first notion and
generalizations of it are commonly used.

For algebraic varieties and the first notion, Demmel \cite{dem88}
and Beltr\'an \&\ Pardo \cite{belpar07} show upper bounds of the
form $c \cdot \deg X \cdot (\epsilon/B)^{2m}$ on the density of
complex $\epsilon$-tubes inside the $B$-neighborhood of $0$, where
$c$ does not depend on $\epsilon$ or $B>\epsilon$. In the real case,
$2m$ is replaced by $m$. Often it is sufficient to consider $B=1$.
Lower bounds, with various values for $c$, are also available.
\cite{sma81} uses the third notion and shows in his Theorem 4A an
upper bound of $k(\epsilon/B)^{^{2}}$ for the hypersurface of  monic
squareful univariate polynomials of degree $k$ in $\C^{k}$. These
papers investigate the condition number which is large for inputs at
which (iterative) numerical methods behave badly, such as the
matrices close to singular ones for matrix inversion or the
(univariate) polynomials close to squareful ones for Newton's root
finding method. These hypersurfaces are also the topic of
\cite{shusma93}.

Yet another notion is the $2d$-dimensional volume of a
$d$-dimensional variety in a complex affine space. Demmel
\cite{dem88}, Section 7, provides bounds on this volume.

We study the density of tubes around  the (affine closed, usually
reducible) variety of decomposable univariate polynomials. An
isomorphism with an affine space (Theorem \ref{geometry}) suggests a
preferred (constant) direction in which to attach
$\epsilon$-neighborhoods, thus following the third one of the
recipes sketched above.

Cheung, Ng \&\ Tsang \cite{cheng13} also consider decomposable
polynomials. They bound the density of a hypersurface containing
them, using the third notion. This provides, by necessity, a weaker
bound than ours.

This paper is organized as follows. Section \ref{sec:nt} presents a
decomposition algorithm which is central for our approach, and
results on the dimensions and degrees of various varieties of
decomposable polynomials. Section \ref{heightBound} discusses bounds
on the growth of coefficients in the decomposition algorithm
mentioned above. These bounds are asymptotically optimal in a
certain sense. Section \ref{section: density estimates in R} defines
our tubes over $\R$ and presents
upper and lower bounds for the resulting density (Theorem
\ref{th:least}). Section \ref{densityComplex} considers the
analogous problem over $\C$. Section \ref{discus} takes up the above
discussion of other notions of tubes and of related work, with some
more detail.

\section{The Newton-Taylor decomposition algorithm}
\label{sec:nt}

It is well known that we may assume all three polynomials in
\eqref{comp} to be monic (leading coefficient 1) and original
(constant coefficient 0, so that the graph contains the origin). All
other compositions can be obtained from this special case by
composing (on the left and on the right) with linear polynomials
(polynomials of degree 1); see, e.g., von zur Gathen \cite{gat13}.

Thus we consider for a proper divisor $d$ of $n$ and $e=n/d$
\begin{align}
\label{allPolys}
P_n(F) = & \; \{ f \in F[x] \colon \deg f = n, f \text{ monic original}\}, \nonumber\\
\gamma_{n,d} \colon & P_d(F) \times P_e(F) \rightarrow P_n(F) \text{
  with }
\gamma_{n,d} (g,h) = g \circ h, \nonumber \\
C_{n,d}(F) = & \; \{ f \in P_n(F) \colon \exists g,h \in P_d(F)
\times
P_e(F) \;\; f = g \circ h \}  \nonumber \\
= & \; \text{im } \gamma_{n,d} , \nonumber \\
C_n(F) = & \bigcup_{{d \mid n} \atop {d \not \in \{1,n\} } }
C_{n,d}(F).
\end{align}
$P_n(F)$ is an $(n-1)$-dimensional vector space over $F$ and
$C_n(F)$ is the algebraic variety of decomposable polynomials. We
drop the argument $F$ when it is clear from the context. When $n$ is
prime, then $C_n(F)$ is empty, and in the following we always assume
$n$ to be composite.

We recall the decomposition algorithm from von zur Gathen
\cite{gat90c}.  It computes $\gnd^{-1}(f)$ for $f \in {\cnd}$ by
taking the reverse $\tilde{f} = x^{n} \cdot f(x^{-1})$ of $f$ and
computing its $d$th root $\tilde{h}$ modulo $x^{e}$, via Newton
iteration with initial value $1$. Thus $\tilde h ^d \equiv \tilde f
\bmod x^e$, $\deg \tilde h < e$, and $\tilde h (0) = 1$. Then the
reverse $h = x^{e} \cdot \tilde{h}(x^{-1})$ of $\tilde{h}$ is monic
original of degree $e$ and is the unique candidate for the right
component. The Newton iteration is well-defined unless $~char(F)$
divides $d$. Like any polynomial of degree at most $n = de$, $f$ has
a (unique) generalized Taylor expansion $f = \sum_{0 \leq j \leq
  d} G_j h^j$ around $h$, with all $G_j \in F[x]$ of degree less than
$e$. Then $f \in C_{n,d}$ if and only if all $G_j$ are constants,
and if so, indeed $f = g \circ h$ with $g = \sum_{0 \leq j \leq d}
G_j x^j$. We call this the Newton-Taylor (NT) method for
decomposing. This computation expresses each coefficient of $g$ and
$h$ as a polynomial in the coefficients of $f$, as illustrated in
Example \ref{exa:ghf}. It can be executed with $O(n \log^{2}n
\log\!\log n)$ operations in $F$. For more details on the computer
algebra machinery, see Section \ref{heightBound}, the cited article,
and von zur Gathen \&\ Gerhard \cite{gatger13}, Sections 9.2 and
9.4.

When $f$ is known to be in $\cnd$ and $h$ has been calculated from
its $e$ highest coefficients, then only the coefficients of $f$ at
powers $x^i$ with $d$ dividing $i$ are needed to compute $g$. We let
$N = \{1, 2, \ldots, n-1\}$ be the support of a general $f \in P_n$.
The NT method only uses the coefficients $f_i$ of $f$ at $x^i$ for
those $i \in N$ which are in the \emph{Newton-Taylor set}
\begin{equation*}
\label{nt} \nt_d = \{ n-1, \ldots, n-e + 1 \} \cup \{ i \in N \colon
e \mid i \}.
\end{equation*}
We also take the complement $\cnt_d = N \setminus \nt_d$. Then
\begin{align}
\label{dim}
\# \nt_d  & = n/d -1 + d -1 = d+n/d-2,\nonumber \\
m_d & = \# \cnt_d = n -d -n/d +1.
\end{align}

\begin{figure}[h!]
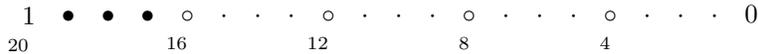

  \centering
  \def\s#1#2{\mathop{#1}\limits_{\strut\llap{$\scriptstyle #2$}}}%
  \def\bb#1{\s\bullet{}}%
  \def\gb#1{\s\circ{#1}}%
  \def\nd#1{\s\cdot{}}%
  $$
  \begin{array}{ccccccccccccccccccccc}
    \s{1}{20} & \bb{} & \bb{} & \bb{} &
    \gb{16} & \nd{} & \nd{} & \nd{} &
    \gb{12} & \nd{} & \nd{} & \nd{} &
    \gb{\phantom{0}8} & \nd{} & \nd{} & \nd{} &
    \gb{\phantom{0}4} & \nd{} & \nd{} & \nd{} &
    0
  \end{array}
  $$
  ~
  \caption{The Newton-Taylor set for $n=20$ and $d=5$.}
  \label{fig:coeffs}
\end{figure}

\begin{example}\label{exa:ghf}
  For $n = 20$, $d = 5$, and $e = 4$, we have $NT_{5} = \{ 19, 18, 17,
  16, 12, 8, 4 \}$ and $\#NT_{5} = d + n/d -2 = 7$.  The leading and
  trailing coefficients of any $f\in P_{20}$ are fixed as $1$ and $0$,
  respectively.  The bullets and open circles in Figure \ref{fig:coeffs} are
  positions of coefficients used in the Newton and Taylor algorithms,
  respectively.  Using the binomial expansion (also known to Newton)
  instead of the Newton iteration and $u= f_{19} x + f_{18} x^{2} +
  f_{17} x^{3}$, we find
\begin{align*}
  \tilde{h} & = 1 + h_{3} x + h_{2} x^{2} + h_{1}x^{3} \equiv (1 +
  u)^{1/5}\\
  & = \sum_{\ell \geq 0}\binom{1/5}{\ell} u^{\ell} \equiv 1 + \frac{1}{5} u -
  \frac{2}{25} u^{2} + \frac{6}{125} u^{3}
  \mod x^{4},\\
  h & = x^{4} + h_{3} x^{3}+ h_{2} x^{2} + h_{1}x\\
  & = x^{4} + \frac{f_{19}}{5} \cdot x^{3} + \frac{-2 f_{19}^{2} + 5
    f_{18}}{25} \cdot x^{2} + \frac{6 f_{19}^{3} - 20 f_{18} f_{19} +
    25 f_{17}}{125} \cdot x.
\end{align*}
Now
$$
g_{4} = f_{16} - \frac{1}{125}\; (21 f_{19}^{4} - 90 f_{18}
f_{19}^{2} + 50 f_{18}^{2} + 100 f_{17} f_{19})
$$
is the coefficient of $x^{16}$ in $f-h^{5}$. Similarly, $g_{j}$ for
$j = 3,2,1$ is determined as the coefficient of $x^{4j}$ in $f -
\sum_{j
  < k \leq 5} g_{k} h^{k}$. These four coefficients of $g$ have
degrees $4$, $8$, $12$, and $16$, respectively, in the
$\nt_{5}$-coefficients of $f$.

To complete the picture, one can express the twelve
$\cnt_{5}$-coefficients of $f \in C_{20,5}$ as polynomials in the
seven $\nt$-coefficients. The polynomials have the following degrees
at $x^{15}$, $x^{14}$, $x^{13}$, $x^{11}, \ldots$: $
5,5,7,9,10,11,13,14,15,17,18,19. $
The B\'{e}zout number, that is, the product of these degrees, is
much larger
than the bound in Theorem \ref{geometry} below.
\end{example}

Using these facts, we give a geometric description of $\cnd$.

\begin{theorem}
  \label{geometry} Let $d$ be a proper divisor of $n$ and assume that
  $~char(F)$ does not divide $d$. Then $\cnd(F) = ~im
  \gamma_{n,d}$ is a closed irreducible algebraic subvariety of
  $P_n(F)$ of dimension $d+n/d-2$ and codimension $m_d$.  The
  Newton-Taylor method provides a polynomial section $\nu_{n,d} \colon
  \cnd \rightarrow P_d(F) \times P_{n/d}(F)$ of $\gamma_{n,d}$.
  Furthermore, $\gamma_{n,d}$ and $\nu_{n,d}$ are defined over $\Z$
  and $\Z[d^{-1}]$, respectively.  For $f = \sum_{1 \leq i \leq n} f_i
  x^i \in \cnd$, $\nu_{n,d} (f)$ depends only on the coefficients
  $f_i$ with $i \in \nt_d$. The degree of $\cnd$ is at most
  $d^{d+n/d-2}$.
\end{theorem}

\begin{proof}
Let $e=n/d$. To show that $\cnd$ is closed and irreducible, we embed
$P_n(F) = \{ (f_{n-1},\ldots, f_1) \in {\mathbb A}^{n-1}(F) \}$ in
$\Pp^{n-1}(F)$ via
$$
(f_{n-1},\ldots, f_1) \mapsto \bar f = (1:f_{n-1}:\ldots : f_1).
$$
The projective version $\bargnd$ of $\gnd$ is then
\begin{align*}
\bargnd & ((g_d: g_{d-1}:\ldots : g_1), (h_e: h_{e-1}\colon \ldots \colon h_1)) \\
& = \sum_{1\leq j \leq d} g_j h_e^{d-j} \bar h^j \in \Pp^{n-1}(F),
\end{align*}
where $\bar h = \sum_{1\leq \ell \leq e} h_\ell x^\ell$ and, with a
slight abuse of notation, $\bar h^j$ stands for the (projective)
vector of $n$ coefficients of the polynomial $\bar h^j$. This vector
is homogeneous in the variables $h_\ell$ of degree $j$.

The scalar extension of $\bargnd$ to
$\Pp^{d-1}(\overline{F})\times\Pp^{e-1}(\overline{F})\to
\Pp^{n-1}(\overline{F})$ is well--defined, where $\overline{F}$ is
an algebraic closure of $F$. As this extension is a closed mapping
which is defined over $F$, we conclude that $\bargnd$ is also a
closed mapping. In particular, $\barcnd = \text{im } \bargnd$ is
closed in $\Pp^{n-1}(F)$.

We now show that $\barcnd (F) \cap {\mathbb A}^{n-1}(F) = \cnd (F)$.
The inclusion \linebreak $\cnd(F) \subseteq \barcnd (F) \cap
{\mathbb A}^{n-1}(F)$ is clear. For the other inclusion, we take
some $f \in \barcnd (F) \cap {\mathbb A}^{n-1}(F)$ and $\bar g \in
\Pp^{d-1}(F)$ and $\bar h \in \Pp^{e-1}(F)$ with $\bar f = \bar g
\circ \bar h$. Since $f \in {\mathbb A}^{n-1}(F) = P_n(F)$, the
leading coefficient of $\bar f$ (at $x^n$) is nonzero. We normalize
$\bar f$ so that this coefficient equals $1$. The coefficient of $
\bar g \circ \bar h = \sum_{1\leq j \leq d} g_j h_e^{d-j} \bar h^j$
at $x^n$ equals $g_d \cdot ~lc (\bar h^d) = g_d \cdot h_e^d$. It
follows that $g_d h_e \neq 0$. After normalizing $\bar g$ and $\bar
h$ by dividing by their leading coefficients, we obtain polynomials
$g \in P_d(F)$ and $h \in P_e(F)$ with $f = g \circ h$.  This shows
the desired inclusion and the claim that $\cnd$ is closed in
$P_{n}$. Furthermore, as $P_d(F)\times P_e(F)$ is irreducible, it
follows that $\cnd = \text{im } \gnd$ is also irreducible.

We prove the degree estimate. The existence of the section
$\nu_{n,d}$ implies that $\dim \cnd = d+e-2$. Let $H_1,\dots,
H_{d+e-2}$ be hyperplanes of $P_n(F)$ with $$\#(\cnd\cap
H_1\cap\cdots\cap H_{d+e-2})=\deg \cnd.$$
Let $\mathcal{S}=\cnd\cap
H_1\cap\cdots\cap H_{d+e-2}$. Then $\#S =\deg C_{n,d}$ and%
$$\gnd^{-1}(\mathcal{S})=\gnd^{-1}(H_1)\cap\cdots\cap \gnd^{-1}(H_{d+e-2}).$$

The polynomial map $\gamma_{n,d}$ consists of $n-1$ integer
polynomials in the coefficients of $g$ and $h$, all of total degree
at most $d$.  Furthermore, for each $i \leq d+e -2$ there exists a
linear combination $w_i$ of the polynomials which define the
coordinates of $\gnd$ so that $\gnd^{-1}(H_i)=\{w_i=0\}$. Therefore,
$\deg\gnd^{-1}(H_i)\le d$ and, by the B\'ezout inequality (see,
e.g., Heintz \cite{hei83}, Fulton \cite{ful84}, Vogel \cite{vog84}),
it follows that \linebreak $\deg \gnd^{-1}(\mathcal{S})\le
d^{d+e-2}$. Let $\gnd^{-1}(\mathcal{S})=\bigcup_{1 \leq j \leq
k}X_j$ be the decomposition of $\gnd^{-1}(\mathcal{S})$ into
irreducible components. Since
$\gnd(\gnd^{-1}(\mathcal{S}))=\mathcal{S}$ and each irreducible
component $X_j$ of $\gnd^{-1}(\mathcal{S})$ is mapped by
$\gnd$ to a point of $\mathcal{S}$, we deduce that%
\begin{align*}
\deg \cnd=\#\mathcal{S}&\le k\\&\le\sum_{1 \leq j \leq k} \deg
X_j\\&=\deg \gnd^{-1}(\mathcal{S})\\&\le d^{d+e-2}.
\end{align*}

The Newton and Taylor algorithms are integral algorithms, except
that divisions by $d$ occur in Newton iteration.
\end{proof}

This precise description of $\cnd$, with the section  $\nu_{n,d}$,
is the basis for our bounds on the real and complex densities that
we consider.

The convex function $d+n/d$ of $d$ assumes its maximum among the
proper divisors of $n$ at $d=l$ and $d=n/l$, where $l$ is the least
prime number dividing $n$; see von zur Gathen \cite{gat13}. Thus the
two ``large'' components of $\cn$ are $C_{n,l}$ and $C_{n,n/l}$,
unless $n = l^{2}$, when they coincide. We will deal with the other
components of smaller dimension at the end of Section \ref{section:
density estimates in R}.

We also want to show that the sum of the two ``large'' densities
bounds the density of $\cn$ from below. To this end, it suffices to
prove that the intersection of the two components has small
dimension. Since both are irreducible, it suffices to show that they
are distinct. This follows easily from Ritt's Second Theorem.  In
fact, the following geometric variant of the normal form for Ritt's
Theorem in von zur Gathen \cite{gat14} provides precise bounds.
\begin{theorem}\label{th:nde}
\label{ritt2} Let $n$, $d$, and $e=n/d$ be as above, with $e > d
\geq 2$ in addition, $i = \gcd(d,e)$, $s= \lfloor e/d \rfloor$ and
$F$ a field of characteristic either $0$ or coprime to $n$.  Then
$X=C_{n,d}(F) \cap C_{n,e}(F) $ is a closed algebraic subvariety of
$P_n(F)$.

When $d \geq 3i$, then $X$ has exactly two irreducible components,
one of dimension $ 2i+s-1 $ and another one of dimension $2i$, and
the intersection of the two is irreducible of dimension $2i-1$.
When $d \leq 2i$, then $X$ is irreducible. It has dimension $d+e/d -
3/2$ if $d=2i$, and dimension $2d + e/d - 3$ if $d = i$.

In all cases, $\dim X < \dim C_{n,d}(F) = \dim C_{n,e}(F)$.
\end{theorem}

\begin{proof}
We first assume that $d \geq 2i$.  Theorem 6.3 from von zur Gathen
\cite{gat14} provides, expressed in geometric language, two
polynomial functions
\begin{align*}
\alpha_\text{exp} \colon & P_i \times F^s \times F \times P_i \rightarrow P_n, \\
\alpha_\text{trig} \colon & P_i \times F \times F \times P_i
\rightarrow P_n,
\end{align*}
so that $X = \text{im } \alpha_\text{exp} \cup \text{im }
\alpha_\text{trig}$. Here \emph{exp} stands for exponential and
\emph{trig} for trigonometric collisions. More precisely,
\begin{align*}
  \alpha_\text{exp} (u,w,a,v) & = u \circ (x^{d(e-sd)/i^2} w^{d/i}(x^{d/i}))^{[a]} \circ v, \\
  \alpha_\text{trig} (u,z,a,v) & = u \circ T_{n/i^{2}}(x,z)^{[a]} \circ v.
\end{align*}
Here $T_n$ is the \emph{Dickson polynomial} of degree $n$, closely
related to the Chebyshev polynomial and satisfying $T_n(x,0) = x^n$.
The monic (but possibly not original) polynomial $w = \sum_{0 \leq i
\leq s} w_i x^i$ of degree $s$ corresponds to the vector $(w_{s-1},
\ldots, w_0) \in F^s$, with $w_s =1$. The \emph{original shift}
$p^{[a]}$ of a polynomial $p \in F[x]$ by $a \in F$ is
$$p^{[a]}=(x-p(a)) \circ p \circ (x+a).$$
Furthermore, $\alpha_\text{exp}$ and $\alpha_\text{trig}$ are
injective, $\dim \text{im }\alpha_\text{exp} = 2i+s-1,$ and
\linebreak $\dim\text{im }\alpha_\text{trig} = 2i$. If $d=2i$ then
$\text{im } \alpha_\text{trig} \subseteq \text{im
}\alpha_\text{exp}$. Otherwise we have $d \geq 3i$ and
\begin{align*}
\text{im } \alpha_\text{exp } \cap \text{ im } \alpha_\text{trig} &
=
P_{i} \circ (x^{n/i^{2}})^{[F]} \circ P_{i} \\
& = \{u \circ ((x + a)^{n/i^{2}} -a^{n/i^2}) \circ v \colon u,v \in
P_{i}, a \in F\}
\end{align*}
has dimension $2i-1$. For the dimension inequality in this case, we
have $2i+s-1 \leq d+ e/2-1 \leq d+e-3$.

When $d = i$, the same Theorem 6.3 shows that $C_{n,d} \cap C_{n,e}
= P_{i} \circ P_{e/d} \circ P_{i}$ is irreducible of dimension
$2i+s-3 = 2d+e/d-3 < d+e-2 =\dim C_{n,d}$.  For more details, see
the cited paper.
\end{proof}

The main point here is that the dimension of this intersection is
less than the dimension of its two arguments.

\section{Bounding the height of $f$, $g$, and $h$}
\label{heightBound} For  the lower bounds of Theorems
\ref{th:prop-div} and \ref{th:composite} below, we analyze the
coefficient growth in the Newton-Taylor method. We start by making
more explicit the form of $f$, $g$, and $h$ in terms of the
Newton-Taylor coefficients of $f = g \circ h$. Recall that $n = d
\cdot e$ are the degrees.

Following the approach of Section \ref{sec:nt}, we consider $$u =
\sum_{1\leq i < e} u_i x^i = \sum_{1 \leq i < e} f_{n-i} x^i\equiv
\tilde f -1 \mod x^e$$ and $$v = \sum_{0\leq \ell < e} v_\ell
x^\ell\ \textrm{with }v^d \equiv 1+u \bmod x^e\textrm{ and
}v(0)=1.$$
The binomial expansion, as in Example \ref{exa:ghf}, says that
\begin{equation}
\label{binomialExpansion}
 v = (1+u)^{1/d}
= \sum_{0\leq \ell < e} \binom {1/d} \ell u^\ell \bmod x^e.
\end{equation}
Then we call $h = x^e \cdot v(x^{-1})$  the \emph{reverse} of $v$.
This differs slightly from the usual reverse $x^{\deg v} \cdot
v(x^{-1})$, since $\deg v < e$.

The Taylor iteration determines the coefficients of $g = \sum_{1
\leq j \leq d} g_j x^j$ as follows. We have $g_d = 1$, and for $j=
d-1, d-2, \ldots, 1$,
 $g_j$ is the coefficient of $x^{ej}$ in $f- \sum_{j<k\leq d} g_k h^k$.
Finally, $f = \sum_{1 \leq i \leq n} f_i x^i = g \circ h = \sum_{1
\leq j \leq d} g_j h^j$.

%
We first determine the ``degrees'' of $h$, $g$, and $f$ in terms of
the NT-coefficients of $f$. More precisely, if we consider the
coefficients $f_i$ with $i \in \nt_d$ as variables, then the
coefficients of $f$ (with $i\in\cnt_d$), $g$, and $h$ are
polynomials over $R= \Z[d^{-1}]$ in these variables.  In order to
make this rigorous, we introduce the generic polynomial $F^{\nt_d} =
\sum_{i \in \nt_d} F_i x^i$ in the set $\varnt=\{ F_i \colon i \in
\nt_d \}$ of indeterminates.

Since all three polynomials are monic, we also set $F_n = G_d = H_e
=1 \in \Z$; these are not indeterminates. We imitate the above
equations, but now in the new indeterminates rather than the
coefficients of $f$.
\begin{equation}
\label{polyDef}
\begin{aligned}
U   = & \sum_{1\leq i < e} U_i x^i  = \sum_{0 \leq i < e} F_{n-i} x^i, \\
V  = &  \sum_{0\leq \ell < e} \binom {1/d} \ell U^\ell \bmod x^e, \\
H  = &  \sum_{1\leq \ell \leq e} H_\ell x^\ell = \text{ reverse of } V, \\
G  =  & \sum_{1\leq j \leq d} G_j x^j, \text{ where }  \\
& \quad G_j = F_{ej}
 - (\text{coefficient of }
 x^{ej} \text{ in }  \sum_{j<k\leq d} G_k H^k ), \\
F  =  &  \sum_{1\leq i \leq n} F_i x^i =  \sum_{1 \leq j \leq d} G_j
H^j = G \circ H.
\end{aligned}
\end{equation}
All five quantities are polynomials in $R[x, \varnt]$, where $R = \Z
[d^{-1}]$; for $V$, this follows from the formula for linear Newton
iteration, which involves division only by $d$.
$F$, $G$, and $H$ are monic original, and each of their coefficients
$H_\ell$, $G_j$, and $F_i$ is a  polynomial in $R[\varnt]$. The
$G_j$ are well-defined for $j = d, d-1, \ldots, 1$, and all $F_i$
that occur in the first and fourth equation have $i \in \nt_d$. It
is convenient to express the degrees under consideration using the
(unusual) grading $~gr(F_i) = n-i$ for $i \in \nt_d$.
\begin{proposition}\label{fgh-degBound}
The polynomials are homogeneous of the following grades.
\begin{enumerate}
\item
$~gr(U_i) = i$ for $1 \leq i < e$,
\item
$~gr( V_\ell) = \ell$ for $0 \leq \ell < e$,
\item
$~gr (H_\ell) =  e-\ell \text{ for } 1 \leq \ell \leq e$,
\item
$~gr (G_j) = n-ej \text{ for } 1 \leq j \leq d$,
\item
$~gr (F_i) = n-i \text{ for } 1 \leq i \leq n$ with $i \not\in
\nt_d$.
\end{enumerate}
\end{proposition}
\begin{proof}
(ii) We have $V_0 =1$, and
 for $\ell \geq 1$, $V_\ell$ is the coefficient of $x^\ell$ in the sum defining $V$.
 Since $U$ is divisible by $x$, this coefficient is a sum of terms
 $U_{m_1} U_{m_2} \cdots U_{m_r}$ with $r \leq \ell$,
 positive integers $m_1, \ldots, m_r$, and
 ${m_1} + {m_2} + \cdots + {m_r} = \ell$, where we leave out the binomial coefficients.
 In particular, $~gr (V_\ell) =\ell$ or $V_\ell=0$.
 Furthermore, $F_{n-1}^\ell$ occurs in $U_1^\ell$ and in $V$ with nonzero coefficient $\binom {1/d} \ell$.

(iii) In the reverse $H = x^e \cdot V(x^{-1})$ of $V$, we have $~gr
(H_\ell) = e-\ell$.

(iv) The claim is shown by downward induction on $j$ from $d$ to $
1$. For $j=d$, we have $G_d = 1$, of grade $0$. For $j<d$, the
contribution of $\varnt$ to $G_j$ has grade $n-ej$. A summand $G_k
H^k$ contributes terms of the form
$$
G_k \cdot H_{\ell_1} \cdots H_{\ell_k}
$$
with positive integers $\ell_1, \ldots, \ell_k$ and $\ell_1 + \cdots
+ \ell_k = ej$. The grade of such a term equals the sum of the
grades of its factors, which by induction is $$n-ek + (e-\ell_1) +
\cdots + (e-\ell_k) = n-ej.$$ Finally, $F_{ej}$ occurs in $G_j$ with
nonzero coefficient $1$.

(v) $F_i$ is the coefficient of $x^i$ in $\sum_{1 \leq j \leq d} G_j
H^j $. Similarly as for (iv), a summand $G_j H^j$ contributes terms
of the form
$$
G_j \cdot H_{m_1} \cdots H_{m_j}
$$
with  $m_1 + \cdots + m_j = i$. The grade of such a term is $$n-ej +
(e-m_1) + \cdots (e-m_j) = n-i.$$ Furthermore, as $F_{ej}$ occurs
only in the summand $G_jH^j$, cancelation cannot occur unless
$H^j=0$ for every $j$.
\end{proof}

%

For a polynomial $f = \sum_i f_i x^i \in \C[x]$ with all $f_i \in
\C$, we consider its height (or infinity norm) $\| f \| = \max_i
|f_i|$. For $f \in \cnd$, we denote as $f^{\nt_d}\in \C^{m_d}$ the
vector of those coefficients of $f$ whose index is in $\nt_d$, and
by $\| f^{\nt_d}\| $ its norm.
 Proposition \ref{fgh-degBound} (v) implies that $\| f \|
= O( \|f^{\nt_d} \| ^n )$. Next we bound the coefficient implicit in
this estimate. We start by considering $u$ and $v$.

\begin{lemma} \label{h-bound}
Let $d, e \geq 2$ be integers, $0 \leq \ell <e$, $A\geq 2$ a real
number, $u,v \in \C[x]$ with $u(0)=0$, $v(0)=1$, $\deg u, \deg v <
e$, $\| u \| \leq A$, and $v^d \equiv (1+u) \bmod x^e$. Then the
following hold.
\begin{enumerate}
\item
$\| u^\ell \| \leq (e A)^\ell$,
\item
$v_0 = 1$ and $|v_l | \leq  (eA)^\ell $ for $\ell \geq 1$.
\end{enumerate}
\end{lemma}
\begin{proof}
(i) Since $u= \sum_{1\leq i <e} u_i x^i$ is a sum of at most $e-1$
summands, the expansion of $u^\ell$ contains not more than
$(e-1)^\ell < e^\ell$ summands, each of which is absolutely at most
$A^\ell$.

(ii) Equation (\ref{fgh-degBound}) holds for $v$, and $v_\ell$ is
the coefficient of $x^\ell$ in
\begin{equation}
\label{sumForV}
 v = (1+u)^{1/d}
= \sum_{0\leq m < e} \binom {1/d} m u^m
= 1 + \frac {u_1} d x + O(x^2) . 
\end{equation}
Now $\binom {1/d} 0 =1$, and for $m \geq 1$ we have
$$
\begin{aligned}
\biggl | \binom {1/d} m \biggr | & = \biggl | \frac {(1/d) \cdot (1/d-1) \cdots (1/d - m+1) } {m!} \biggr |
\\[1ex]
& = \biggl | \frac {1 \cdot (1-d) \cdots (1-(m-1)d)} {d^m \cdot m!}
\biggl | \, \\[1ex]&< \frac {d \cdot 2d \cdots (m-1)d} {d^m \cdot m!}
\\&= \frac 1 {d m}.
\end{aligned}
$$
For $C \geq 4$, we have $\sum_{1 \leq m \leq \ell} C^m/m \leq 2
C^\ell/\ell$, as follows by induction on $\ell \geq 1$. Since $d,e
\geq 2$ and $u^m$ with $m>\ell$ does not contribute to $v_\ell$, we
also have
\begin{align*}
| v_\ell | &\leq \sum_{1\leq m \leq \ell} \frac 1 {dm} \|u^m \|
\\&\leq \frac 1 2 \sum_{1\leq m \leq \ell} \frac 1 m (e A )^m \\&\leq
\frac {(e A)^\ell} \ell.
\end{align*}
\end{proof}


\begin{proposition}\label{fgh-constBound}
Let $d,e \geq 2$ be integers, $f, g, h \in \C[x]$ be monic original
polynomials of degrees $n=de$, $d$, $e$, respectively, $f = g \circ
h$, and $\| f^\nt \| \leq A$ with $A \geq 2$. Then the following
hold.
\begin{enumerate}
\item
$| h_\ell | \leq  (eA)^{e-\ell}  \text{ for } 1 \leq \ell \leq e$,
\item
$| g_j | \leq  e^{(d-j)(d+j+1)/2} (eA)^{n-ej} \text{ for } 1 \leq j
\leq d$,
\item
$| f_i | \leq 2 e^{d(d+1)/2} (eA)^{n-i} \text{ for } 1 \leq i \leq
n$ with $i \not\in \nt_d$.
\end{enumerate}
\end{proposition}

\begin{proof}
Since $h$ is the reverse of $v$, Lemma \ref{h-bound} (ii) implies
the claim (i).

For (ii), we have  $g_d=1$ and
$$
g_j = \text{ coefficient of } x^{ej} \text{ in } f-\sum_{j<k\leq d}
g_k h^k
$$
for $1\leq j < d$. Since $ej \in \nt_d$, the required coefficient of
$f$ occurs in $f^\nt$. We first bound the  coefficient of $ x^{i}$
in $ h^k = (\sum_{1\leq \ell \leq e} h_\ell x^\ell )^k$ for $1 \leq
i \leq n$. It is the sum of terms $h_{m_1} \cdots h_{m_k}$ with
integers $1 \leq m_1, \ldots, m_k \leq e$ and $m_1 + \cdots + m_k
=i$. By (i), each such term is bounded in absolute value by $$(e
A)^{(e-m_1) \cdots (e-m_k)}
 = (e A)^{ek-i}.$$
Since $h$ is a sum of $e$ monomials, there are at most $e^k$ such
terms that contribute to the coefficient in question. Except for
$i=n$ and $k=d$, the choice $m_1 = \cdots = m_k =e$ does not
contribute. Thus the absolute value of this coefficient of $x^i$ is
at most
\begin{equation}
\label{coeffxiBound} (e^k-1) (e A)^{ek-i}.
\end{equation}
For $1 \leq j <d$, we have with $i=ej$ that
\begin{equation}
\label{firstGbound}
| g_j |  \leq A  + \sum_{j<k\leq d} | g_k |  (e^k-1) (e A)^{ek-ej}.
\end{equation}

In (ii),  we claim that
\begin{equation}
\label{gbound} | g_j | \leq b_j (eA)^{n-ej}
\end{equation}
 with
$b_j =  e^{(d-j)(d+j+1)/2}$. We have $g_d = b_d = 1$ and $b_{d-1} =
e^d$. First, we show  that $$\sum_{j<k\leq d} b_k (e^k-1) < b_j$$ by
downward induction for $j=d-1, \ldots, 1$. For $j=d-1$, we have
$e^d-1 < e^d = b_{d-1}$. For $j< d-1$, we find
$$
\begin{aligned}
 \sum_{j<k\leq d} b_k (e^k-1)
& =  \sum_{j+1<k\leq d} b_k (e^k-1) + b_{j+1}( e^{j+1}-1)  \\
& < b_{j+1} + b_{j+1}( e^{j+1}-1) \\&= b_j.
\end{aligned}
$$

We will absorb the lonely term $A$ in (\ref{firstGbound}) into the
summand for $k=d-1$. First, we note that
\begin{equation}
\label{2bound} \frac 2 {e-1} + 2 \leq (2e)^e,
\end{equation}
since $e\geq 2$. (In fact, (\ref{2bound}) holds for $e \geq 1.55$.)
This implies that
\begin{equation}
\label{Aebound} \frac A {(e^{d-1}-1)(eA)^{n-ej}} + \frac A {(eA)^e}
\leq 1
\end{equation}
for $j < d$, since  (\ref{2bound}) is the special case $d=2$,
$j=d-1$, $A=2$ of (\ref{Aebound}), the left hand side of
(\ref{Aebound}) is monotonically decreasing in $d$ and $A$ and
increasing in $j$, and the special case takes the extreme values of
$d$, $j$, and $A$ under our assumptions. In turn, this means that
\begin{equation*}
\label{gdbound} \frac A {(e^{d-1}-1)(eA)^{n-ej}} +
(A+(e^d-1)(eA)^e)  (eA)^{-e} \leq e^d = b_{d-1}.
\end{equation*}
By (\ref{firstGbound}), we have $|g_{d-1} | \leq A +(e^d-1)(eA)^e$,
and thus
\begin{equation}
\label{gd-1bound}
\begin{aligned}
& A +  | g_{d-1} | (e^{d-1}-1)(eA)^{n-e-ej} \\
& \leq (e^{d-1}-1)(eA)^{n-ej} \big( \frac A {(e^{d-1}-1)(eA)^{n-ej}} \\
& \quad + (A + (e^d-1)(eA)^e) (eA)^{-e} \big) \\
& \leq
b_{d-1} (e^{d-1}-1)(eA)^{n-ej}.
\end{aligned}
\end{equation}
We finally prove (\ref{gbound}) by downward induction for $j=d,d-1,
\ldots, 1$, using (\ref{firstGbound}) and (\ref{gd-1bound}). The
cases where $j\in \{d, d-1\}$ are clear. For $j<d-1$, we separate
the summands for $k\in\{d,d-1\}$ and find
$$
\begin{aligned}
| g_j | &  \leq A  + \sum_{j<k\leq d} | g_k |  (e^k-1) (e A)^{ek-ej} \\
& = A + (e^d-1) (eA)^{n-ej} + | g_{d-1} | (e^{d-1}-1) (eA)^{n-e-ej} \\
& \quad + \sum_{j<k\leq d-2} | g_k |  (e^k-1) (e A)^{ek-ej} \\
& \leq  \sum_{j<k\leq d} b_k (eA)^{n-ek}  (e^k-1) (e A)^{ek-ej} \\
& = (eA)^{n-ej}  \sum_{j<k\leq d} b_k (e^k-1) \\&< b_j (eA)^{n-ej}.
\end{aligned}
$$

(iii) Since $f_i$ is the coefficient of $x^i$ in $\sum_{1\leq j \leq
d} g_j h^j$, we find from (\ref{coeffxiBound}) that
$$
\begin{aligned}
| f_i | \leq & \sum_{1\leq j \leq d}  e^{(d-j)(d+j+1)/2} (eA)^{n-ej}
\cdot e^j (eA)^{ej-i} \\
= &  \, e^{d(d+1)/2} (eA)^{n-i} \sum_{1\leq j \leq d} e^{-j(j-1)/2}.
\end{aligned}
$$
The exponents in the sum are pairwise different, so that its value
is at most the complete geometric sum, of value $e/(e-1) \leq 2$.
\end{proof}

According to Proposition \ref{fgh-degBound}, 
all $f_i$ (with $i\in\cnt_d$), $g_j$, and $h_\ell$ are homogeneous
of grades $n-i$, $n-ej$, and $e-\ell$, respectively. This implies
that the exponents of $A$ in Proposition \ref{fgh-constBound} cannot
be improved. However, the exponents of $e$ are less precisely
determined.

\begin{corollary}
\label{finalBound} Let $A, B \geq 2$ and $f = g \circ h \in \cnd$
with $g$ and $h$ monic original of degrees $d$ and $e$,
respectively, and $\|f ^\nt \| \leq A \leq {B^ {1/n}}/
{e^{1+(d+1)/2e}} $. Then the following bounds hold.
\begin{enumerate}
\item
$\| h \| \leq  (eA)^{e-1}$,
\item
$\| g \| \leq  e^{(d(d+1)/2-1} (eA)^{n-e}$,
\item
$\| f \| \leq 2 e^{d(d+1)/2} (eA)^{n-1} < e^{d(d+1)/2} (eA)^{n} \leq
B$.
\end{enumerate}
\end{corollary}

\section{Density estimates for $\cnd(\R)$ and $\cn(\R)$}
\label{section: density estimates in R} In this section we consider
the set $P_n(\R)$ of monic original polynomials of composite degree
$n$ with real coefficients and the subsets $\cn(\R)$ and $\cnd(\R)$
of $P_n(\R)$ for a proper divisor $d$ of $n$. Our aim is to obtain
density estimates on tubes around $\cnd(\R)$ and $\cn(\R)$.  We drop
the field $F = \R$ from our notation in this section.

We identify $P_n$ with $\R^{n-1}$ by mapping
$x^n+a_{n-1}x^{n-1}+\dots +a_1x \in P_n$ to $(a_{n-1},\dots,a_1) \in
\R^{n-1}$.  As shown above, $\cnd$ is an affine real variety of
dimension $d+{n}/{d}-2$. In particular, $\cnd$ has codimension at
least $n/2$, and thus its  (standard Lebesgue) volume is 0.  For a
meaningful concept, we take a specific $\epsilon$-tube around
$\cnd$. Namely, for each $f = \sum_{1\leq i
  \leq n} f_i x^i \in \cnd$ and $\epsilon > 0$, we define the
$\epsilon$-neighborhood of $f$ as
\begin{align*}
\label{real neighbor} U_\epsilon (f) = \bigl\{ u =  & \sum_{1\leq i
\leq n} u_i x^i \in P_n(\R) \colon
 u_i = f_i \text{ for } i \in \nt_d, \\
& |u_i - f_i| < \epsilon \text{ for } i \in \cnt_d \bigr\}.
\end{align*}
Thus $U_\epsilon (f)$ is an open $m_d$-dimensional hypercube $(-
\epsilon, \epsilon)^{m_{d}}$ in $P_n$. Around each coefficient $f_i$
with $i \in \cnt_d$ we have a real interval of length $2\epsilon$.
We also set
\begin{equation}\label{eq:cnd}
U_\epsilon (\cnd) = \bigcup_{f \in \cnd} U_\epsilon (f).
\end{equation}

In order to have finite volumes, we take a bound $B>0$ on the
coefficients and consider the $(n-1)$-dimensional hypercube
$$P_{n,B} =\bigl\{f=\sum_{1\le i\le n}f_ix^i\in P_n \colon |f_i|<B \text{ for } 1\leq i < n
\bigr\}$$ around $P_{n}$ and its intersection with the
$\epsilon$-tube
$$
U_{\epsilon,B}(\cnd)=U_\epsilon (\cnd)\cap P_{n,B}.
$$
Our main purpose is to obtain estimates on the density
$~den_{\epsilon,B}(C_{n,d})$ of the $\epsilon$-tube in $P_{n,B}$,
namely
\begin{equation}\label{eq: definition density of cnd(R)}
  ~den_{\epsilon,B}(C_{n,d})  = \frac{~vol
    \left(U_{\epsilon,B} (\cnd)
      \right)}{~vol
    (P_{n,B})}= \frac{~vol \left(U_{\epsilon,B} (\cnd)
      \right)}{(2B)^{n-1}}.
\end{equation}

In a slightly different model of our situation, we might allow
arbitrary leading coefficients in our polynomials, rather than just
$1$. It would then be sufficient to just consider the unit hypercube
with $B=1$ and scale the resulting density. However, our approach is
overall more convenient and allows an easier comparison with
previous work; see Section \ref{discus}.

\begin{example}\label{exa:subspace}
  For perspective, we calculate the density of the linear subspace $L
  = \R^{k}\times \{0\}^{n-k} \subseteq \R^{n}$. For $x \in L$, we take
  $U_{\epsilon}(x) = \{(x_{1}, \ldots, x_{k})\} \times (-\epsilon,
  \epsilon)^{n-k}$ and have, for $B > \epsilon$,
\begin{align}\label{al:UepsiL}
\begin{aligned}
  U_{\epsilon}(L)& = \R^{k} \times (-\epsilon, \epsilon)^{n-k},\\
 ~den_{\epsilon, B}(L)& = \frac{~vol\big((-B, B)^{k} \times (-\epsilon,
    \epsilon)^{n-k}\big)}{(2B)^{n}} = \left(\frac{\epsilon}{B}\right)^{n-k}.
\end{aligned}
\end{align}
\end{example}

Let $\chi \colon P_{n,B}\to\{0,1\}$ be the characteristic function
of
 $U_{\epsilon,B} (\cnd) \subseteq P_{n,B}$. Then
 the density is
$$~den_{\epsilon,B}(C_{n,d})
=\frac{1}{(2B)^{n-1}}\int_{P_{n,B}}\chi(a)\, \mathrm{d}a.$$

For a subset $S\subseteq N$ of cardinality $s$, we consider
the projection $\pi^S \colon \R^{n-1}\to\R^s$ onto the coordinates
in $S$: $\pi^S(a_{n-1},\dots,a_1)=(a_i \colon i\in S)$. Furthermore,
for a subset $C\subseteq P_n$, we write $C^S$ for $ \pi^S (C)$. By
reordering the coordinates, we can express
 the hypercube $P_{n,B}$ as the Cartesian product
$P_{n,B}=P_{n,B}^{\nt_d}\times P_{n,B}^{\cnt_d}$. According to
Fubini's theorem we have
\begin{align}
\begin{aligned}
~vol \bigl(U_{\epsilon,B} (\cnd)\bigr)&=
\int_{P_{n,B}}\chi(a) \, \mathrm{d} a \\
&=\int_{P_{n,B}^{\nt_d}}\Big(\int_{P_{n,B}^{\cnt_d}}
\chi(a^{\nt_d},a^{\cnt_d}) \, \mathrm{d}a^{\cnt_d} \Big)\,
\mathrm{d}a^{\nt_d}.
 \label{eq: volume via Fubbini th}
\end{aligned}
\end{align}
Here $(a^{\nt_d}, a^{\cnt_{d}})$ are the coordinates of $a \in
P_{n,B}$ in the product representation, and $a^{\nt_{d}}$ refers to
a point in $P^{\nt_{d}}_{n,B}$.

\begin{lemma}\label{lemma: upper bound iterated integral}
  Let $0 < \epsilon < B$ and $b \in P_{n,B}^{\nt_d}$.
  \begin{enumerate}
\item
\label{lemma: upper bound iterated integral-1} We have
$$
\int_{P_{n,B}^{\cnt_d}}\chi(b, c) \, \mathrm{d}  c  \le
(2\epsilon)^{m_d},
$$
where $c$ ranges over $P_{n,B}^{\cnt_{d}}$.
\item
\label{lemma: upper bound iterated integral-2} Let $f\in\cnd$ be the
unique element with $\pi^{\nt_d} (f) = b$. If
$U_{\epsilon}(f)\subseteq P_{n,B}$, then equality holds in
\eqref{lemma: upper bound iterated integral-1}.
  \end{enumerate}
\end{lemma}

\begin{proof}
The existence and uniqueness of $f$ follow from the section
$\nu_{n,d} \colon$ $\cnd \to P_d \times P_{n/d} $ (Theorem
\ref{geometry}). For any $c \in P_{n,B}^{\cnt_{d}}$, we have
$$
\chi (b,c) = 1 \Longleftrightarrow (b,c) \in U_{\epsilon,B}(\cnd)
\Longleftrightarrow (b,c) \in U_{\epsilon}(f) \cap P_{n,B}.
$$
Therefore
\begin{align*}
  \int_{P_{n,B}^{\cnt_d}}\chi(b, c) \, \mathrm{d} c
&\le  \int_{U_\epsilon(f)} \mathrm{d}   c =( 2\epsilon)^{m_d}. 
\end{align*}
If $U_\epsilon (f) \subseteq P_{n,B}$, then equality holds. This
shows both claims.
\end{proof}

We derive the following upper bound on the density of $\cnd$.
\begin{proposition}\label{prop: upper bound density}
With notation and assumptions as above, we have
%
$$~den_{\epsilon,B}(\cnd) \le \left(\frac{\epsilon}{B}\right)^{m_d}.$$
\end{proposition}
\begin{proof}
Combining \eqref{eq: volume via Fubbini th} and Lemma \ref{lemma:
upper bound iterated integral}, we obtain
\begin{align*}
~vol \left(U_{\epsilon,B}(\cnd)\right)&=
\int_{P_{n,B}}\chi(a) \,\mathrm{d}a\nonumber\\
&=\int_{P_{n,B}^{\nt_d}}\left(\int_{P_{n,B}^{\cnt_d}}
\chi(a_{\nt_d},a^{\cnt_d})\mathrm{d}a^{\cnt_d}
\right)\mathrm{d}a^{\nt_d}\nonumber\\
&\le\int_{P_{n,B}^{\nt_d}}(2\epsilon)^{m_d} \mathrm{d}a^{\nt_d}\\
&=(2\epsilon)^{m_d}(2B)^{n-1-m_d}.
\end{align*}

Using \eqref{eq: definition density of cnd(R)},
it follows that
\begin{align*}
~den_{\epsilon,B}(\cnd)
\le\frac{(2\epsilon)^{m_d}(2B)^{n-1-m_d}}{(2B)^{n-1}}
=\left(\frac{\epsilon}{B}\right)^{m_d}.
\end{align*}
\end{proof}

Now we derive a lower bound. We set
\begin{equation}
\label{alpha} \alpha_{n,d} = (d/n)^{1+d(d+1)/2n}
\end{equation}
 and $A_B = \alpha_{n,d} B^{1/n}$.
From Corollary \ref{finalBound} (iii), we know that if $B\ge
(2/\alpha_{n,d})^n$ and $f\in\cnd$ is such that $f^{\nt_d}\in P_{n,
A_{B-\epsilon}}^{\nt_d}$, then $U_\epsilon(f)\subset
U_{\epsilon,B}(\cnd)$.

\begin{proposition}\label{prop: lower bound density}
With notation and assumptions as above, and assuming that
$B-\epsilon \geq (2/\alpha_{n,d})^n$, we have
\begin{align*}
~den_{\epsilon,B}(\cnd) \ge \left(\frac{\epsilon}{B}\right)^{m_d}
\bigg(\frac{\alpha_{n,d} (B-\epsilon)^{1/n}} B
\bigg)^{d+\frac{n}{d}-2}.
\end{align*}
\end{proposition}

\begin{proof}
We let
 $$
V = \bigcup_{\{f \in \cnd : f^{\nt_d}\in
P_{n,A_{B-\epsilon}}^{\nt_d}\} }
  U_{\epsilon}(f) \subseteq U_{\epsilon,B}(\cnd)
$$
and $\chi_\epsilon\colon P_{n,B}\to\{0,1\}$ be the characteristic
function of $V$. Since $V \subseteq U_{\epsilon,B}(\cnd)$, it
follows that $\chi_\epsilon(a)\le\chi(a)$ for every $a\in P_{n,B}$.
Using Lemma \ref{lemma: upper bound iterated integral}\eqref{lemma:
upper bound iterated integral-2},
we find

\begin{align*}
  \int_{P_{n,B}}\chi(a)\, \mathrm{d}a&\ge
  \int_{P_{n,B}}\chi_\epsilon(a)\, \mathrm{d}a\\
  & =\int_{P_{n,A_{B-\epsilon}}^{\nt_d}}  \Big(\int_{P_{n,
        B}^{\cnt_d}}
    \chi_\epsilon(a^{\nt_d},a^{\cnt_d}) \, \mathrm{d}a^{\cnt_d}
  \Big)\mathrm{d}a^{\nt_d} \\
  & =
\int_{P_{n,
A_{B-\epsilon}}^{\nt_d}}(2\epsilon)^{m_d}\mathrm{d}a^{\nt_d}\\
&=(2\epsilon)^{m_d}\big(2A_{B-\epsilon}\big)^{n-1-m_d}.
\end{align*}
Dividing by $~vol (P_{n,B})=(2B)^{n-1}$ and using \eqref{dim} yields
the claimed bound.
\end{proof}

We summarize the results of Propositions \ref{prop: upper bound
  density} and \ref{prop: lower bound density} as follows.
\begin{theorem}\label{th:prop-div}
  Let $0 < \epsilon < B$ be as in Proposition \ref{prop: lower bound density} and let $d$ be a proper divisor of $n$.
  Then we have the following bounds on the density of the
  $\epsilon$-tube around $C_{n,d}(\R)$:
$$
 c_d(\epsilon,B)\cdot\left(\frac{\epsilon}{B}\right)^{n-d-\frac{n}{d}+1} \le
~den_{\epsilon,B}(C_{n,d}(\R))\le
\left(\frac{\epsilon}{B}\right)^{n-d-\frac{n}{d}+1},
$$
where $c_d(\epsilon,B)=\big(\frac{(B-\epsilon)^{\frac{1}{n}}}{B}
(\frac{d}{n})^{1+\frac{d(d+1)}{2n}}\big)^{d+\frac{n}{d}-2}$.
\end{theorem}

We thus have good bounds on the irreducible components of $C_{n}$ in
\eqref{allPolys}.  How to get such bounds for $C_{n}$ itself? In
Theorem \ref{th:prop-div}, we consider $\epsilon$-tubes around
$\cnd$ of a direction and a dimension that varies with $d$. For
$C_{n}$, it seems appropriate to consider $\epsilon$-tubes of the
same dimension for all $d$, as follows.

We let $l$ be the least prime number dividing the composite integer
$n$.  If $n = l^2$, then $\cn = C_{n,l}$ has just one component.
Otherwise, Theorem \ref{geometry} shows that $\cn(\R) \subseteq
\R^{n-1}$ has two ``large'' components, namely $C_{n,l}$ and
$C_{n,n/l}$, each of dimension $l + n/l-2 = ~dim C_{n}$.  As a
consequence, for the density of $\cn$ we consider $\epsilon$--tubes
of the same dimension $m_l=n-1-\dim \cn$ around each component of
$\cn$. For a proper divisor $d\not\in \{l,n/l\}$ of $n$, we have
$\dim \cnd< \dim\cn$ and thus $\dim \cnd+m_l<n-1$. Then any
$m_l$--dimensional $\epsilon$--tube around $\cnd$ has volume and
density equal to zero. Furthermore, Theorem \ref{th:nde} implies
that the sum of the two ``large'' densities bounds the density of
$\cn$ from below.  In other words, we define $$~vol_{\epsilon,
B}(C_{n}) = ~vol(U_{\epsilon, B}(C_{n,l}) \cup U_{\epsilon, B}(C_{n,
n /l}))$$ and $$~den_{\epsilon,
  B}(C_{n}) = ~vol_{\epsilon, B}(C_{n})/ ~vol(P_{n,B}).$$
Setting
\begin{align}\label{al:delta}
\delta_n =
 \begin{cases}
1  & \text{if } n = l^2,\\
2 & \text{otherwise},
\end{cases}
\end{align}
we obtain the following result.
\begin{theorem}\label{th:least}
Let  $0 < \epsilon<B$ be such that $B-\epsilon\ge
(2/\alpha_{n,d})^n$ and let $l$ be the least prime number dividing
the composite integer $n$. Then
$$\delta_nc_l(\epsilon,B) \left(\frac{\epsilon}{B}\right)^{n-l-n/l+1}
\le ~den_{\epsilon,B}(\cn(\R))\le \delta_n
\left(\frac{\epsilon}{B}\right)^{n-l-n/l+1}.$$
\end{theorem}

The approximation factor $c_l(\epsilon,B)$ in the lower bound
tends to a constant smaller than $1$, depending only on
$n$, when $\epsilon/B$ gets small compared to $n$.

\section{Density estimates for $\cnd(\C)$ and $\cn(\C)$}
\label{densityComplex}
In this section, we take $F = \C$ and consider the real volume on
$P_n(\C)$ as a $(2n-2)$-dimensional real vector space. We discuss
briefly the density estimates we obtain for $\cnd(\C)$. The approach
is similar to that of Section \ref{section: density estimates in R};
therefore, we merely sketch the proofs and summarize the results we
obtain. We drop the field $F = \C$ from our notation.

As in  \eqref{eq:cnd}, we take an $\epsilon$-tube around $\cnd$,
namely given $f  = \sum_{1\leq i \leq n} f_i x^i \in \cnd$ and
$\epsilon > 0$, we define the (complex) $\epsilon$-neighborhood of
$f$ as
\begin{align*}
U_\epsilon (f) = \Big\{ u =  & \sum_{1\leq i \leq n} u_i x^i \in P_n
\colon
 u_i = f_i \text{ for } i \in \nt_{d}, \\
& |u_i - f_i| < \epsilon \text{ for } i \in \cnt_{d} \Big\}.
\end{align*}
Thus $U_\epsilon (f)$ is an open $m_d$-dimensional complex
polycylinder in $P_n$, of real dimension $2m_{d}$. Around each
coefficient $f_i$ with $i \in \cnt_{d}$, we have a real circle of
radius $\epsilon$ and area $\pi
\epsilon^2$. 
For $B > 0$, we set
\begin{align*}
U_\epsilon (\cnd) & = \bigcup_{f \in \cnd} U_\epsilon (f),\\
P_{n,B} & = \Big \{ f = \sum_{1\leq i \leq n} f_i x^i  \in P_{n}
\colon |f_i| < B
\text{ for } 1 \leq i < n \Big\}, \\
U_{\epsilon,B} (\cnd)  & = U_\epsilon(\cnd )\cap P_{n,B}.
\end{align*}

Then $~vol (P_{n,B}) = (\pi B^2)^{n-1}$.  Let $\chi \colon P_{n,B}
\rightarrow \{ 0,1\}$ be the characteristic function of
$U_{\epsilon,B} (\cnd)$.  As before, we express the polycylinder
$P_{n,B}$ as the Cartesian product $P_{n,B}=P_{n,B}^{\nt_d}\times
P_{n,B}^{\cnt_d}$ and apply Fubini's theorem to obtain
\begin{align}
  ~vol \left(U_{\epsilon,B}(\cnd)\right)=\int_{P_{n,
      B}^{\nt_d}}\Big(\int_{P_{n,B}^{\cnt_d}}
    \chi(a^{\nt_d},a^{\cnt_d})\, \mathrm{d}a^{\cnt_d}
  \Big)\, \mathrm{d}a^{\nt_d}. \label{eq: volume via Fubbini th
    complex}
\end{align}

For an arbitrary element $b \in P_{n,B}^{\nt_d}$, there exists a
unique $f\in \cnd $ with $\pi^{\nt_{d}}(f) = b$ by Theorem
\ref{geometry}. Then the function $\chi( b, f^{\cnt_d})$ takes the
value $1$ on an $m_{d}$-dimensional complex polycylinder of radius
$\epsilon$ whose center is the vector of coefficients of $f$
corresponding to indices in $\cnt_d$. As a consequence, we have
\begin{align*}
~vol \left(U_{\epsilon, B}(\cnd)\right)& =
\int_{P_{n,B}^{\nt_d}}(\pi \epsilon^2)^{m_d}\mathrm{d}a^{\nt_d}
\\&\leq (\pi \epsilon^2)^{m_d} (\pi B^2)^{n-1-m_d}.
\end{align*}
This yields the complex analog of the upper bound of Proposition
\ref{prop: upper bound density}:
$$
~den_{\epsilon, B}(\cnd)= \frac{~vol (U_{\epsilon,B}(\cnd))}{(\pi
  B^{2})^{n-1}} \leq \big(\frac{\epsilon}{B}\big)^{2m_{d}}.
$$
 On the other hand, for a lower
bound we consider as in Proposition \ref{prop: lower bound density}
the characteristic function $\chi_\epsilon \colon P_{n,B}\to\{0,1\}$
of the set
$$
V = \bigcup_{f \in \cnd \cap P_{n,A_{B-\epsilon}}}U_{\epsilon}(f)
\subseteq U_{\epsilon, B}(\cnd)
$$
for $\epsilon<B$, where
$A_{B-\epsilon}=\alpha_{n,d}(B-\epsilon)^{\frac{1}{n}}$ and
$\alpha_{n,d}$ is defined as in \eqref{alpha}, and argue as before
to obtain
$$
~vol (U_{\epsilon, B} (\cnd) ) \geq ~vol (V) \geq
(\pi\epsilon^{2})^{m_d} (\pi
(\alpha_{n,d}(B-\epsilon)^{\frac{1}{n}})^{2})^{d+\frac{n}{d}-2},
$$
provided that $B-\epsilon\ge (2/\alpha_{n,d})^n$.

Finally, in order to obtain a meaningful notion of density of $\cn$,
we consider, as in Section \ref{section: density estimates in R},
the $\epsilon$--tube $U_{\epsilon, B}\left(\cn \right)=U_{\epsilon,
B}\left(C_{n,l}\right) \cup U_{\epsilon, B} \left(C_{n,n/l} \right)$
around $\cn$, where $l$ is the smallest prime divisor of $n$.
Summarizing, we have the following results on the density of these
tubes.
 \begin{theorem}\label{th:composite}
 \label{complexBound}
Let $0 <  \epsilon <B$, $n$ be a composite integer, and $\delta_{n}$
as in \eqref{al:delta}.
 \begin{enumerate}
 \item Let $d$ be a proper divisor of $n$ and $B-\epsilon\ge (2/\alpha_{n,d})^n$. Then
 \begin{align*}
   c_d'(\epsilon,B)\left( \frac \epsilon B \right) ^{2(n-d- \frac{n}{d} +1)}& \leq
   ~den_{\epsilon,B}(\cnd(\C))
     \leq \left( \frac \epsilon B \right) ^{2(n - d - \frac{n}{d} +1)},
\end{align*}
where $c_d'(\epsilon,B)=\big((B-\epsilon)^{\frac{1}{n}}
B^{-1}(\frac{d}{n})^{1+\frac{d(d+1)}{2n}}\big)^{2(d+\frac{n}{d}-2)}$.\smallskip
 \item
 Let $l$ be the smallest prime divisor of $n$ and $B-\epsilon\ge (2/\alpha_{n,l})^n$. Then
\begin{align*}
\delta_{n} c_l'(\epsilon,B) \left( \frac{\epsilon}{B}\right)
^{2(n-l-\frac{n}{l}+1)} \le ~den_{\epsilon,B}(\cn(\C))
  \leq \delta_{n} \left( \frac \epsilon B \right) ^{2(n-l-\frac{n}{l}+1)}.
\end{align*}
%
 \end{enumerate}
 \end{theorem}

\section{Discussion}
\label{discus}

For an arbitrary irreducible algebraic subvariety $X$ of $\R^n$ with
codimension $m$, we might attach a hypercube
$(-\epsilon,\epsilon)^m$ to each smooth point $x$ of $X$ in the
normal direction to $X$, thus following the second recipe listed in
the introduction.  The singular points do not contribute to the
volume.  Then this tube around $X$ has real dimension $n$.  In an
analog of Lemma \ref{lemma: upper bound iterated
  integral}, the coordinates in $\nt_d$ are replaced by local
coordinates at the point and those of $\cnt_d$ by coordinate
functions normal to them.  Instead of having a unique $f$ as in the
proof of that lemma, we only know that the normal linear space, of
complementary dimension, intersects generically in at most $\deg X$
points. The resulting upper bound then is $ \deg X \cdot (2
\epsilon)^m$.

Our construction in \eqref{eq:cnd} of the $\epsilon$-tube around
$\cnd$ does not follow this general recipe, since the
$\epsilon$-hypercube in the direction of the coordinates from
$\cnt_{d}$ is, in general, not normal to $\cnd$. It is not clear
whether one can obtain
 upper and lower bounds as in Theorems \ref{th:least} and
\ref{th:composite} for other choices of the $\epsilon$-tubes.

Cheung \emph{et al.} \cite{cheng13} also provide bounds on the
density of $\cn(\C)$.
 Instead of the precise information provided by the Newton-Taylor
 method of Section \ref{sec:nt}, they use the fact that $\cn(\C) \subseteq X$
 for a certain hypersurface $X$ and then a specific one-dimensional
 $\epsilon$-tube around $X$ chosen to suit their argument,
 following the third of the options listed in the introduction. They show
 that $~den_{\epsilon,B}(\cn(\C))\leq (n^2-2n) \cdot (\epsilon/B)^2$,
 which is to be compared with our result in Theorem \ref{complexBound}.

\section{Acknowledgements}
Many thanks go to Igor Shparlinski for alerting us to the paper of
Cheung \emph{et al.} \cite{cheng13}.


\providecommand{\bysame}{\leavevmode\hbox
to3em{\hrulefill}\thinspace}
\providecommand{\MR}{\relax\ifhmode\unskip\space\fi MR }
\providecommand{\MRhref}[2]{%
  \href{http://www.ams.org/mathscinet-getitem?mr=#1}{#2}
} \providecommand{\href}[2]{#2}

\end{document}